\documentclass[12pt]{article}

\usepackage{latexsym,amssymb,amsmath,bm}

\newcommand{\C}{\mathbb C}
\newcommand{\Z}{\mathbb Z}
\newcommand{\N}{\mathbb N}
\newcommand{\F}{\mathbb F}

\newcommand{\Q}{\mathbb Q}
\newcommand{\R}{\mathbb R}

\newcommand{\sma}{\left(\begin{array}}
\newcommand{\fma}{\end{array}\right)}

\newtheorem{lem}{Lemma}[section]
\newtheorem{defn}[lem]{Definition}

\newtheorem{co}[lem]{Corollary}
\newtheorem{thm}[lem]{Theorem}
\newtheorem{prop}[lem]{Proposition}

\newenvironment{proof}{\textbf{Proof.}}{\newline\hspace*{\fill}{$\Box$}\\}

\begin{document}
\title{Aspects of non positive curvature for linear groups with no
infinite order unipotents}

\author{J.\,O.\,Button\\
Selwyn College\\
University of Cambridge\\
Cambridge CB3 9DQ\\
U.K.\\
\texttt{j.o.button@dpmms.cam.ac.uk}}
\date{}
\maketitle

\begin{center}
\end{center}
\hfill

\begin{abstract}
We show that a linear group without unipotent elements
of infinite order
possesses properties akin to those held by groups of non positive
curvature. Moreover in positive characteristic any finitely generated
linear group acts properly and semisimply on a CAT(0) space.
We present applications, including that the mapping
class group of a surface having genus at least 3
has no faithful linear representation which is complex
unitary or over any field of positive characteristic.
\end{abstract}

\section{Introduction}

Knowing that a group $G$ is non positively curved allows us to draw strong
conclusions about $G$, not just about its geometry but also its group 
theoretic structure. This is explained in the book \cite{bdhf} where two 
notions of what it means for a group to be non positively curved are examined
in detail. First is the class of CAT(0) groups, which are those groups
acting geometrically
(namely properly and cocompactly by isometries) on some CAT(0) 
metric space $(X,d)$, whereupon we can conclude:\\
\begin{thm} (\cite{bdhf} Part III Chapter $\Gamma$
Theorem 1.1 Part 1)\\ \label{know}
A CAT(0) group $\Gamma$ has the following properties:\\
(1) $\Gamma$ is finitely presented.\\
(2) $\Gamma$ has only finitely many conjugacy classes of finite
subgroups.\\
(3) Every solvable subgroup of $\Gamma$ is virtually abelian.\\
(4) Every abelian subgroup of $\Gamma$ is finitely generated.\\
(5) If $\Gamma$ is torsion-free, then it is the fundamental group of
a compact cell complex whose universal cover is contractible.
\end{thm}

Now the focus of this paper is on linear groups $G$ (here meaning
that $G$ embeds in the general linear group $GL(d,\F)$ for some
integer $d$ and some field $\F$ of arbitrary characteristic) which
are finitely generated. Such groups also have good properties, for instance
they are residually finite and satisfy the Tits alternative. However
in general there is no relationship between linearity and non positive
curvature. For instance the
Burger-Mozes groups in \cite{bm}
are CAT(0) and even have a geometric action on a
2 dimensional CAT(0) cube complex, but are infinite simple finitely
presented groups, so are as far from being linear as can be imagined. 
As for the other way round, it is not hard to see using well
known examples in small dimension that
none of the five properties above are true for every finitely generated linear
group: taking
various subgroups of $SL(2,\Z)\times SL(2,\Z)$ is enough to show
that (1), (2), (5) do not hold and the Heisenberg group of upper
unitriangular 3 by 3 matrices over $\Z$ fails (3), whereas 
wreath products
such as $\Z\,\wr\,\Z$ in $GL(2,\R)$, or in positive characteristic we could
take $C_p\,\wr\,\Z$ in $GL(2,\F_p(t))$, do not satisfy (4).

Another source of variation is that linear groups have good closure
properties, such as being preserved under taking subgroups and
commensurability. As for CAT(0) groups, even a finitely presented
subgroup of a CAT(0) group need not itself be CAT(0) (see \cite{bdhf}
Chapter III $\Gamma$ Section 5) and it is unknown whether being CAT(0)
is preserved under commensurability. However there is a more general
notion of non positive curvature: here we will say that a group $G$
is {\bf weak CAT(0)} if $G$ has an
isometric action on a complete CAT(0) space $X$
which is proper and semisimple
(meaning that for any $g\in G$ the displacement function $x\mapsto
d(x,g(x))$ attains its infimum over $X$).
For this notion of non positive curvature (which does indeed hold
if we have a CAT(0) group) we obtain:
\begin{thm} (\cite{bdhf} Part III Chapter $\Gamma$ Theorem 1.1 Part 2)\\
\label{know2}
If $H$ is a finitely generated group that acts properly (but not necessarily
cocompactly) by semisimple isometries on the CAT(0) space $X$, then:\\
(i) Every polycyclic subgroup of $H$ is virtually abelian.\\
(ii) All finitely generated abelian subgroups of $H$ are
undistorted in $H$.\\
(iii) $H$ does not contain subgroups of the form
$\langle a,t|t^{-1}a^pt=a^q\rangle$ 
for non zero $p,q$ with $|p|\neq |q|$.\\
(iv) If $A\cong\Z^n$ is central in $H$ then there exists a subgroup
of finite index in $H$ that contains $A$ as a direct factor.\\
\end{thm}

Moreover this class of weak CAT(0) groups is closed under taking
subgroups and commensurability. But once again none of the properties
(i) to (iv) hold for all finitely generated linear groups: for instance
the Heisenberg group above does not satisfy (i), (ii) or (iv)
whereas the Baumslag-Solitar group $BS(1,2)$ fails (ii) and (iii)
but embeds in $GL(2,\Q)$. 
However in writing out the obvious matrices that generate these
counterexamples, one is struck by how often unipotent matrices 
(where all eigenvalues equal 1) appear. This might lead us to ask: what if
we only consider linear groups with no unipotent elements (other than the
identity)? Do such groups share the properties (i) to (iv)
of non positive curvature in Theorem \ref{know2}? Indeed are
they even weak CAT(0) groups?
In this paper we will consider both zero and positive characteristic
linear groups. In fact over positive characteristic it is straightforward
to see that all unipotent elements have finite order, whereas our
counterexamples above used infinite order unipotents. In the next
section we prove Theorem \ref{pos}, which states that any finitely generated
linear group in positive characteristic is weak CAT(0) and so we conclude
by the above theorem that it does
possess all of the properties (i) to (iv).

In characteristic zero we also consider the class of linear
groups with no infinite order unipotent elements, which now is equivalent to
saying there are no non trivial unipotents (as in this case all other unipotent
matrices have infinite order).
This includes any real orthogonal or complex unitary group, as well as
much else besides, and we can show similar properties of non positive
curvature for these groups too. In particular
Section 3 is about splitting theorems for
centralisers and we establish property (iv) in Corollary \ref{coiv}
for finitely generated linear groups in characteristic zero with no
non trivial unipotents.
An immediate application of these results, using the ideas in \cite{bri} on
centralisers of Dehn twists in mapping class groups,
is then given in Corollary \ref{yoh}:
the mapping class group of an orientable  surface of genus at least 3
cannot be linear
over any field of positive characteristic, nor can it embed in the
complex unitary group of any finite dimension.
Indeed any linear representation in any dimension over any field
sends each Dehn twist to a matrix that either has finite order
or is virtually unipotent.

In Section 4 we examine the difference between
a finitely generated group $G$ having all of its infinite cyclic subgroups
undistorted in $G$ and the stronger property where this holds for
all finitely generated abelian subgroups. The relevance
here to linear groups is that it was shown in \cite{lmr} that if $G$ is a
finitely generated linear group with no infinite order unipotents
then any infinite cyclic
subgroup of $G$ is undistorted. This fact also holds under the
same hypotheses for finitely generated abelian subgroups of $G$, but
instead we examine how we might establish this stronger property
without using linearity directly. To this end, we adapt some ideas of
G.\,Connor on translation lengths to show by a quick argument
in Theorem \ref{trids} that if a finitely generated group
$G$ has ``uniformly non distorted'' infinite cyclic subgroups
then any finitely generated abelian subgroup is also undistorted in $G$.
As an application, consider the group $Out(F_n)$ (which is known not to be
a linear group when $n\geq 4$).
It was recently shown in \cite{wrig} that any abelian subgroup of
$Out(F_n)$ (which will be finitely generated) is undistorted. This extended the
paper \cite{alib} which established the same result for
infinite cyclic subgroups of $Out(F_n)$. However as this paper actually
obtains ``uniform non distortion'' for infinite cyclic subgroups,
we obtain as an immediate corollary of Theorem \ref{trids} a much shorter
proof of the result in \cite{wrig}.
We finish by noting that the property of undistorted abelian subgroups then
extends immediately to $Aut(F_n)$ and free by cyclic groups $F_n\rtimes\Z$. 

We would like to thank the various referees for their comments on the initial
submission of this paper, resulting here in a version which is both stronger 
and shorter than the original draft.

\section[linear groups in positive characteristic]{Finitely generated
linear groups in positive characteristic}

Let $k$ be any field in any characteristic and $d\in\N$ any positive integer.
An element $g$ of the general
linear group $GL(d,k)$ is said to be {\bf unipotent}
if all its eigenvalues (considered over the algebraic closure
$\overline{k}$ of $k$) are equal to 1, or equivalently 
some positive power of $g-I$ is the zero matrix. It is then an easy exercise
(for instance using Jordan decomposition)
to show that if $k$ has characteristic $p>0$ then any unipotent
element $M$ in $GL(d,k)$ has finite (indeed $p$ power) order,
whereas if $k$ has zero characteristic then the only finite
order unipotent element in $GL(d,k)$ is $I_d$.

Given an arbitrary field $k$ (again in any characteristic for now),
a {\bf discrete valuation} on $k$ is a function $\nu:k\rightarrow
\Z\cup\{\infty\}$ such that\\
(1) $\nu(x)=\infty$ if and only if $x=0$\\
(2) $\nu(xy)=\nu(x)+\nu(y)$\\
(3) $\nu(x+y)\geq \mbox{min}(\nu(x),\nu(y))$.\\
This gives rise to a non archimedean metric on $k$ via setting
$d(x,y)=e^{-\nu(x-y)}$.
The set of elements ${\cal O}_\nu=\{x\in k:\nu(x)\geq 0\}$ forms a
subring of $k$, called the valuation ring, which is a principal ideal
domain and an element $\pi$ with $\nu(\pi)=1$ is called a uniformiser.

Now suppose that $k$ is a field of characteristic $p>0$ which is
finitely generated over the prime subfield $\F_p$. This means that
there are finitely many algebraically independent transcendental
elements $t_1,\ldots ,t_d$ such that $k$ is a finite extension of the
field $\F_p(t_1,\ldots ,t_d)$. We now adapt Theorem 1 of \cite{eric}
to establish a similar result (in fact they worked in arbitrary
characteristic whereas here we are only in positive characteristic but
we obtain a stronger conclusion for this case).
\begin{prop} \label{vals}
If $k$ is any finitely generated field of positive characteristic and
$R$ is any finitely generated subring of $k$ then there exist
finitely many discrete valuations $\nu_1,\ldots ,\nu_b$ on $k$ such that
for any $m\in\Z$ the set
\[\{r\in R:\nu_i(r)\geq m\mbox{ for all } 1\leq i\leq b\}
\]
is finite.
\end {prop}
\begin{proof}
First suppose this is true for the field $k=\F_p(t_1,\ldots ,t_d)$ and let
us add a new algebraically independent transcendental $t$ to obtain the
field $k(t)$ which contains the unique factorisation domain $k[t]$.
Finite generation of $R$ means that there is a finitely generated
subring $S\subseteq k$ and finitely many monic irreducible polynomials
$p_1,\ldots ,p_u$ (which are prime in $k[t]$) such that $R$ is
contained in the finitely generated ring $S[t,1/p_1,\ldots ,1/p_u]$.
We now proceed by induction on the transcendence degree $d$, so we take
the valuations $\nu_1,\ldots ,\nu_b$ on $k$ satisfying the required
condition for $S$ and we extend each of these to $\nu_1',\ldots ,\nu_b'$
on $k[t]$ and then to the field of fractions $k(t)$ by setting
\[\nu_i'(a_0+a_1t+\ldots +a_nt^n)=
\mbox{min}\left( \nu_i(a_0),\nu_i(a_1),\ldots ,\nu_i(a_n)\right).\]
We also add valuations $\mu_0,\mu_1\ldots ,\mu_u$ given by
$\mu_0(a_0+a_1t+\ldots +a_nt^n)=-n$ and for $1\leq i\leq u$ we have
\[\mu_i(p_i^n\frac{a}{b})=n\in\Z\]
for the primes $p_1,\ldots ,p_u$ above.     

On taking these $b+u+1$ valuations on $k(t)$, suppose we are given
$m\in\Z$. Any $r\in R$ is of the form
\[r=\frac{a}{p_1^{i_1}\ldots p_u^{i_u}}\qquad\mbox{for }a\in S[t]
\mbox{ and }i_1,\ldots ,i_u\geq 0.
\]
If $\nu(r)\geq m$ for every valuation above then we immediately have
$i_1,\ldots ,i_u\leq -m$, giving only finitely many possibilities for
the denominator. Thus we can now consider $r=a$ which is a polynomial
$s_0+s_1t+\ldots +s_nt^n$, but $\mu_0(a)\geq m$ so the degree is also
bounded above by $-m$.

Now our inductive hypothesis is that
\[\{s\in S:\nu_i(s)\geq m\mbox{ for all }1\leq i\leq b\}\]
is finite, and we require polynomials $f(t)=s_0+s_1t+\ldots +s_nt^n$ with
coefficients in $S$ such that $\nu_i'(f)\geq m$ for each $1\leq i\leq b$.
This means that for each coefficient $s_j$, all of $\nu_1(s_j),\ldots ,
\nu_b(s_j)$ are at least $m$ and consequently by our inductive hypothesis
we have only finitely many choices for each of $s_0,\ldots ,s_n$. Thus
overall we have only finitely many choices for $r\in R$.

This concludes the proof for purely transcendental extensions of $\F_p$.
As for finite extensions of these, the argument in \cite{eric} Lemma 3
now goes through verbatim.
\end{proof}

Given any infinite field $k$ with discrete valuation $\nu$, valuation
ring ${\cal O}_\nu$ and a uniformiser $\pi$, we have an action 
of $SL(n,k)$ on its
Bruhat - Tits building, here denoted ${\cal B}_\nu$. We will need the
following facts (see for instance \cite{abbr} Section 6.9):\\
A lattice $L$ in $k^n$ is a free ${\cal O}_\nu$ submodule of $k^n$ that
contains a basis of $k^n$ and a lattice class $[L]$ is the orbit $k^*L$
under the obvious multiplication action. The Bruhat - Tits building
${\cal B}_\nu$ is a contractible $n-1$ dimensional simplicial complex
where the $j$-simplices are subsets of lattice classes
\[\{[L_0],\ldots ,[L_j]\}\mbox{ where }\pi L_j\subset L_0\subset L_1
\subset \ldots \subset L_j.\]
As $SL(n,k)$ acts on the set of lattices and hence also on the set of lattice
classes, it admits a simplicial action on ${\cal B}_\nu$ where the stabiliser
of a vertex is conjugate (in $GL(n,k)$ but not necessarily in $SL(n,k)$) to
$SL(n,{\cal O}_\nu)$. Moreover $SL(n,k)$ acts on ${\cal B}_\nu$ ``without
permutations'', which is to say that given $g\in SL(n,k)$ and a simplex 
$\sigma$ with $g(\sigma)=\sigma$ then $g$ fixes the vertices of $\sigma$
pointwise.

As ${\cal B}_\nu$ has the structure of a Euclidean building, it can be turned
into a metric space by putting the correct Euclidean metric on the top
dimensional simplices (the chambers), whereupon it becomes a complete
CAT(0) space and the simplicial action of $SL(n,k)$ is by isometries.
However the action is not proper (in the sense of a discrete group)
because the stabiliser of a vertex is isomorphic to
the infinite group $SL(n,{\cal O}_\nu)$. In order to  further
examine the stabilisers
of vertices, we use the following basic lemma.
\begin{lem} \label{bas}
Given any vertex $v\in {\cal B}_\nu$, there exists $m\in\Z$ such that for
all elements $g\in SL(n,k)$ fixing $v$ we have $\nu(g_{ij})\geq m$ for every
entry $g_{ij}$ of $g$.
\end{lem}
\begin{proof}
By the fact about stabilisers above, there exists $h\in GL(n,k)$ such that
$hgh^{-1}\in SL(n,{\cal O}_\nu)$. But if we have two matrices $x,y\in GL(n,k)$
and $m,n\in\Z$ such that all entries $x_{ij}$ of $x$ have $\nu(x_{ij})\geq m$ 
and similarly $\nu(y_{ij})\geq n$ then the valuation of any entry of the
matrix product $xy$ is at least $m+n$ by axioms (2) and (3). As each entry
$s_{ij}$ in an element $s$ of $SL(n,{\cal O}_\nu)$ has $\nu(s_{ij})\geq 0$
by definition and $g=h^{-1}(hgh^{-1})h$, we apply this fact twice.
\end{proof}

We can now give our main result.
\begin{thm} \label{pos}
If $G$ is any finitely generated linear group over a field
of positive characteristic then $G$ acts properly and semisimply by
isometries on a complete CAT(0) space.
\end{thm}
\begin{proof} We are given $G\leq GL(n.k)$ where $k$ is an arbitrary
infinite field of positive characteristic but we can increase $n$ by 1 so that
we can assume $G$ is actually a subgroup of $SL(n,k)$. Now finite
generation of $G$ means that the ring $R$ generated by the entries of $G$
is also finitely generated. We then replace $k$ by the field of fractions
of $R$, henceforth also called $k$, thus now $k$ is finitely generated as a
field but is still infinite and of positive characteristic. This means
that $G\leq SL(n,R)\leq SL(n,k)$ and the hypothesis of Proposition
\ref{vals} is satisfied. Consequently we obtain discrete valuations
$\nu_1,\ldots ,\nu_b$ on $k$ with the aforementioned property.

We now show that $SL(n,R)$ acts properly and semisimply by isometries on an
appropriate complete CAT(0) space $X$, thus $G$ does too as this property
passes to arbitrary subgroups. The space $X$ will be the product of the
Bruhat - Tits buildings ${\cal B}_{\nu_1}\times \ldots \times {\cal B}_{\nu_b}$
for each valuation on $k$ that we took. As each of these factors is complete
and CAT(0), our space $X$ will be too on being given the Euclidean product
metric. Moreover $SL(n,k)$ acts on each factor by isometries so it also does so
on $X$ via the diagonal action.

We now need to establish the two properties of this isometric action.
Starting with properness, first consider the stabiliser $T_{\bf v}$ not in
$SL(n,k)$ but rather in $SL(n,R)$ of a product of vertices 
\[{\bf v}=(v_1,\ldots ,v_b)\in {\cal B}_{\nu_1}\times \ldots \times 
{\cal B}_{\nu_b}.\]
Thus if $g$ is in $T_{\bf v}$ so that $g(v_1)=v_1,\ldots ,g(v_b)=v_b$ then
by Lemma \ref{bas} we have integers $m_1,\ldots ,m_b$ with every entry $g_{ij}$
of $g$ having $\nu_1(g_{ij})\geq m_1,\ldots ,\nu_b(g_{ij})\geq m_b$. Hence
for $m=\mbox{min}(m_1,\ldots ,m_b)$ we see that every entry $g_{ij}$ satisfies
$\nu_1(g_{ij})\geq m,\ldots ,\nu_b(g_{ij})\geq m$. This means
that by Proposition \ref{vals}
there are only finitely many possibilities for the entries of $g$ and hence
the stabiliser $T_{\bf v}$ is a finite subgroup.

Now consider the stabiliser $H$ in $SL(n,R)$
of a general point $(x_1,\ldots ,x_b)$. Take any
$h\in H$ with $h(x_1)=x_1, \ldots ,h(x_b)=x_b$ and first consider the action
of $SL(n,R)$ on the Bruhat - Tits building ${\cal B}_{\nu_1}$. Let $d$ be
the dimension such that $x_1$ lies in the $d$-skeleton of ${\cal B}_{\nu_1}$
but not the $(d-1)$-skeleton. Then the $d$-simplex $\sigma_d$ in which $x_1$
lies must be sent to itself by $h$ and so, as $SL(n,R)$ acts without
permutations, we have that $h$ fixes the vertices of $\sigma_d$ pointwise.
Applying this also to $x_2,\ldots ,x_b$, we see that $h$ also fixes a vertex
in every building ${\cal B}_{\nu_i}$ thus $h\in T_{\bf v}$ for some $\bf v$.

In order to obtain properness of the action, we take any point
$(x_1,\ldots ,x_b)\in X$ and consider the simplex $\sigma_d$ for $x_1$ as
above. Then there exists $\epsilon_1>0$ such that the ball $B(x_1,\epsilon_1)$
in ${\cal B}_{\nu_1}$ intersects the $d$-skeleton of ${\cal B}_{\nu_1}$ only
in the interior of the simplex $\sigma_d$. Thus as any element $g\in
SL(n,R)$ sends $\sigma_d$ to some $d$-simplex, if $g$ moves $x_1$ by less 
than $\epsilon_1$ then we have that $g(\sigma_d)=\sigma_d$ and so $g$ fixes
a vertex as above. In particular we obtain $\epsilon_1,\ldots ,\epsilon_b$ and
$\epsilon=\mbox{min}(\epsilon_1,\ldots ,\epsilon_b)$ such that if $g$ moves the
point $(x_1,\ldots ,x_b)$ by less than $\epsilon$ in the product metric on
$X$ then $g$ moves each $x_i$ by less than $\epsilon_i$. Consequently $g$ fixes
a vertex in every building and so lies in the finite subgroup $T_{\bf v}$
for some ${\bf v}$.

Finally to show the action is semisimple, we first quote \cite{bdhf} Part II
6.6 (2) which states that an action by simplicial isometries of a group
on a complete non positively curved simplicial complex having a finite set
of shapes is a semisimple action.
Next we use Proposition 6.9 in the same
part of the same volume, where it is shown that an isometric action
on a product of CAT(0) spaces is semisimple if the
action on each component is too.
\end{proof}   

Note that in general this CAT(0) space is not a proper metric
space because the Bruhat - Tits building associated to the 
valuation $\nu$ of the field $k$
is only locally finite if the residue field of $\nu$ is finite. This
is the case for (finite extensions of) $k=\F_p(t)$ but fails as soon
as we have more than one independent transcendental. 

We finish this section with a few words on related results and on
trying to extend this to the characteristic zero case. Of course there
are fields of characteristic zero possessing discrete valuations, such as
the $p$-adic valuations on $\Q$. In \cite{alsh} a similar result to
our Proposition \ref{vals} in characteristic zero is obtained (though
the resulting sets consist of algebraic integers and need not be
finite) and then the diagonal action on the
resulting product of Bruhat - Tits buildings is utilised to establish
exactly which finitely generated linear groups in characteristic zero
have finite virtual cohomological dimension. Equivalent cohomological
results for linear groups in positive characteristic occur in \cite{cokr} 
which proceeds along the lines of Theorem \ref{pos}, though the stabiliser
of a product of vertices under this action is shown to be locally finite
(which is not strong enough here as finitely generated linear groups in
positive characteristic need not be virtually torsion free).

In order to obtain proper actions of finitely generated linear
groups in characteristic zero, one could also throw in some (non discrete)
archimedean
absolute values (namely embeddings in $\R$ and $\C$) along with
the action on the appropriate symmetric space which will take the place of the
Bruhat - Tits
building, though some care is needed with the archimedean embeddings
if the resulting finitely generated field contains transcendental elements.
However the analogous result in characteristic zero would require that the
action is semisimple whenever our finitely generated linear group has
no non trivial unipotent elements. But by \cite{bdhf} Section II
Proposition 10.61 the elements of $SL(n,\R)$ that act
semisimply on its (real) symmetric space are exactly 
those matrices which are diagonalisable over $\C$, which is certainly
a more stringent restriction even for subgroups of $SL(n,\Z)$.

\section[Centralisers in linear groups]{Abelianisation of centralisers in 
linear groups with no infinite order unipotents}
For groups $H$ which possess some form of non positive curvature,
there are various splitting results in the literature for
centralisers, which generally take the form that if an infinite order
element $h$ (or even a finitely generated free abelian group) is
central in $H$ which is finitely generated then, up to replacing $H$ with a 
finite index subgroup
also containing $h$, we have that $H$ splits as a direct product 
$S\times\langle h\rangle$. This is achieved by considering the
translation part of an element commuting with $h$ when restricted
to an invariant axis of $h$. For our class of finitely generated linear groups
with no infinite order unipotents, we have already shown in Theorem \ref{pos}
that this property holds in positive characteristic and now we do the same in
characteristic zero. We first reduce this to:
\begin{lem} \label{lemqk}
Suppose that $H$ is a finitely generated group and 
$A\cong\Z^n$ is central in $H$. If we have a homomorphism
 $\theta$ from $H$ to some abelian group $C$ which is injective on $A$
then there exists a subgroup
of finite index in $H$ that contains $A$ as a direct factor.
\end{lem}
\begin{proof} By dropping to
the image $\theta(H)$, we can assume that $\theta$ is onto
and so without loss of generality
$C$ is also finitely generated. By the classification
of finitely generated abelian groups, we have that $C=\Z^m\oplus$Torsion
for $m\geq n$ and we can compose $\theta$ with a homomorphism $\phi$ from
$C$ onto $\Z^n$ in which $A$ still injects, so the image $B=\phi\theta(A)$
will have finite index.

Thus if we set $K=\mbox{Ker}(\phi\theta)$ then the pullback 
$(\phi\theta)^{-1}(B)=KA$ has finite index in $H$. Also $K$ and $A$ are
normal subgroups of $H$ with $K\cap A=\{e\}$, giving $KA\cong K\times A$.
\end{proof}

We can now use the determinant as
our homomorphism with abelian image, enabling us to work in
complete generality.

\begin{thm} \label{det}
Suppose that $G$ is a linear group over any field $\F$
of any characteristic
and $A$ is an abelian subgroup which is
central in $G$. (Here neither $G$ nor $A$ is assumed to be
finitely generated.)
Let $\pi$ be the homomorphism from $G$ to
its abelianisation $G/G'$ and $\pi|_A$ the restriction
of $\pi$ to $A$. If $G$, or even $A$, contains no infinite order
unipotent element then
$\mbox{ker}(\pi|_A)$ is a torsion group.
\end{thm}
\begin{proof}
We first replace our field by its algebraic closure, which we will also call
$\F$. Then it is true that any abelian subgroup of $GL(d,\F)$ is conjugate
in $GL(d,\F)$ to an upper triangular subgroup of $GL(d,\F)$, for instance
by induction on the dimension and Schur's Lemma.

For any $g\in G$ and $a\in A$ we have $ga=ag$. This means that
$g$ must map not just each eigenspace of $a$ to itself, but each
generalised eigenspace
\[E_{\lambda}(a)=\{v\in\F^d:(a-\lambda I)^nv=0\mbox{ for some } n\in\N\}
\mbox{ where }\lambda\in\F\]
and together these span, so that if $a$ has distinct eigenvalues
$\lambda^{(a)}_1,\ldots ,\lambda^{(a)}_{d_a}$ then 
$\bigoplus_{i=1}^{d_a} E_{\lambda^{(a)}_i}(a)$ is a
$G$-invariant direct sum of $\F^d$.

We now take a particular (but arbitrary) non identity element $a$ of $A$
and restrict $G$ to the first of these
generalised eigenspaces $E_{\lambda^{(a)}_1}(a)$, so that 
here $a$ only has the one eigenvalue $\lambda^{(a)}_1$.  
If this property also holds on $E_{\lambda^{(a)}_1}(a)$ for every
other $a'\in A$ then we proceed to $E_{\lambda^{(a)}_2}(a)$, 
$E_{\lambda^{(a)}_3}(a)$ and so on. Otherwise there is another $a'\in A$
such that we can split $E_{\lambda^{(a)}_1}(a)$ further into
pieces where $a'$ has only one eigenvalue on each piece.
Moreover this decomposition is also $G$-invariant because it can
be thought of as the direct sum of the generalised eigenspaces of $a'$
when $G$ is restricted to $E_{\lambda^{(a)}_1}(a)$. 

We then continue this process on all of the pieces and over all elements
of $A$ until it terminates (essentially we can view it as building a
rooted tree where every vertex has valency at most $d$ and of finite
diameter). We will now find that we have split $\F^d$ into a $G$-invariant
sum $V_1\oplus\ldots \oplus V_k$ of $k$ blocks, where any element
of $A$ has a single eigenvalue when restricted to any one of these 
blocks. 

Next we conjugate within each of these blocks so that the restriction
of $A$ to this block is upper triangular, using the comment at the
start of this proof. Under this basis so obtained
for $\F^d$, we have that any $a\in A$ will now be of the form
\[a=\sma{ccc}\boxed{T_1}& &0\\
&\ddots&\\
0& &\boxed{T_k}\fma\]
where each block $T_i$ is an upper triangular matrix with all diagonal entries
equal (as these are the eigenvalues of $a$ within this block).
More generally any $g\in G$ will be of the form
\[\sma{ccc}\boxed{M_1}& &0\\
&\ddots&\\
0& &\boxed{M_k}\fma
\]
for various matrices $M_1,\ldots ,M_k$ which are the same size as the
respective matrices $T_1,\ldots ,T_k$ because we know $g$ 
preserves this decomposition.

Consequently we have available as homomorphisms from $G$ to the multiplicative
abelian group $(\F^*,\times)$ not just the determinant itself but also
the ``subdeterminant'' functions
$\mbox{det}_1,\ldots ,\mbox{det}_k$, where for $g\in G$ the
function $\mbox{det}_j(g)$ is defined as
the determinant of the $j$th block of $g$ when
expressed with respect to our basis above, and
these are indeed homomorphisms as is
\[\theta:G\rightarrow (\F^*)^k\mbox{ given by }
\theta(g)=(\mbox{det}_1(g),\ldots ,\mbox{det}_k(g)).\]
As $\theta$ is a homomorphism from $G$ to an abelian
group, it factors through the homomorphism $\pi$ 
from $G$ to its abelianisation because this is the universal abelian
quotient of $G$. This means that $\mbox{ker}(\pi)$ is contained in 
$\mbox{ker}(\theta)$ and so we can replace $\pi$ with $\theta$ for the 
rest of the proof.

Thus suppose that there is some $a\in A$ which is in the kernel of 
$\theta$. We know that
\[a=\sma{ccc}\boxed{T_1}& &0\\
&\ddots&\\
0& &\boxed{T_k}\fma\]
for upper triangular matrices $T_i$
and as all diagonal entries within each $T_i$ are equal, say $\mu_i$
for $\mu_1,\ldots ,\mu_k\in \F^*$,
we conclude that $\mu_i^{d_i}=1$ where $d_i=\mbox{dim}(V_i)$.
In other words $a$ is virtually unipotent (namely some positive power of
$a$ is unipotent), implying under our hypotheses that $a$ has finite order. 
\end{proof}

We now immediately obtain the same
conclusion of Theorem \ref{know2} Part (iv) for finitely generated
linear groups in
characteristic zero, provided only that our abelian subgroup is unipotent
free.
\begin{co} \label{coiv}
If $H$ is any finitely generated linear group in characteristic zero 
and $A\cong\Z^n$ is central in $H$ and does not contain a non identity
unipotent element then
there exists a subgroup
of finite index in $H$ that contains $A$ as a direct factor.
\end{co}
\begin{proof}
Theorem \ref{det} gives us a homomorphism $\theta$ from 
$H$ to some abelian group $C$ whose restriction to $A$ has kernel
consisting only of torsion elements, thus $\theta$
 is injective on $A$ and Lemma \ref{lemqk} applies.
\end{proof}

\subsection{Applications to the mapping class groups}
Of course any finitely generated group $G$ which fails to
satisfy any of the four conditions in Theorem \ref{know2}
cannot act properly  and semisimply by isometries on a CAT(0) space.
An important example of this is the mapping
class group $Mod(\Sigma_g)$ where here
$\Sigma_g$ will be an orientable surface of finite topological type having
genus $g$ at least 3 (which might be closed or might have any
number of punctures or boundary components). 
In \cite{bri} Bridson shows that for all the
surfaces $\Sigma_g$ mentioned above, the mapping class group $Mod(\Sigma_g)$
is not a weak CAT(0) group,
a result first credited to \cite{kaplb}. This is done using the
following obstruction which is similar to Theorem \ref{know2} Part (iv).

\begin{prop} \label{bri} (\cite{bri} Proposition 4.2)\\
If $\Sigma$ is an orientable surface of finite type having genus at
least 3 (with any number of boundary components and punctures) and if $T$
is the Dehn twist about any simple closed curve in $\Sigma$ then the
abelianisation of the centraliser in $Mod(\Sigma)$ of $T$ is finite.
\end{prop}

As this is covered by Theorem \ref{det}, we have
\begin{co} Suppose that 
$\Sigma$ is an orientable surface of finite type having genus at
least 3 (with any number of boundary components and punctures) and
$\rho:Mod(\Sigma)\rightarrow GL(d,k)$ is any linear representation
in any dimension over any field. Then for every Dehn twist
$T\in Mod(\Sigma)$, the matrix $\rho(T)$ either has finite order or is
virtually unipotent.
\end{co}
\begin{proof} Set $A$ and $G$ to be the images under $\rho$ of
$\langle T\rangle$ and the centraliser in $Mod(\Sigma)$ of $T$
respectively. Thus $A$ will be abelian and central in $G$ and we
know by Proposition \ref{bri} that the abelianisation of the
centraliser in $Mod(\Sigma)$ of $T$ is finite, so the abelianisation
of $G$ is also finite. Thus on applying Theorem \ref{det} to $G$ and $A$,
either $A=\langle\rho(T)\rangle$ contains some infinite order unipotent
element or $\mbox{ker}(\pi|_A)$ is a torsion group.
As $\pi$ maps to the finite group $G/G'$, $\rho(T)$ has finite
order in the second case.
\end{proof}

This immediately gives us:
\begin{co} \label{yoh}
If $\Sigma$ is an orientable surface of finite type having genus at
least 3 (with any number of boundary components and punctures)
and $\rho:Mod(\Sigma)\rightarrow GL(d,k)$ is any linear representation
of the mapping class group of $\Sigma$ in any dimension $d$
where the field $k$ either has positive
characteristic or $k=\C$ and the image of $\rho$ lies in the
unitary group $U(d)$ then $\rho(T)$ has finite order for $T$ any
Dehn twist and thus $\rho$ cannot be faithful.
\end{co}

We note that there are ``quantum'' linear representations with
infinite image but where every Dehn twist has finite order.
However 
linearity of the mapping class group in genus $g\geq 3$
is a longstanding open question, although in genus 2 linearity over
$\C$ was established in \cite{bbud} and in \cite{kork}
by applying results on braid groups. 

\section{Undistorted cyclic and abelian subgroups}

Having shown that the non curvature properties in Theorem \ref{know2}
are satisfied by any finitely generated linear group in positive
characteristic, we have also seen that proposition (iv) holds for
finitely generated linear groups in characteristic zero if they contain
no non trivial unipotent matrices. We might also wonder about the
other three properties for this class of groups; indeed they all hold
too (see \cite{mel}). In this section we examine property (ii), which is
that all finitely generated abelian subgroups are undistorted. Our
proof of this in \cite{mel} involves taking a finite number of different
absolute values on the field, in a similar fashion to Proposition
\ref{vals}, and then using the operator norm of the matrix elements
with respect to each absolute value in order to show that there are no
finitely generated abelian subgroups which are distorted. 
In the special case of infinite cyclic
subgroups, this fact was already proved in \cite{lmr} Proposition 2.4
where an argument also using the operator norm was provided and moreover
this is a short proof because only one absolute value is required there.
Now if it were true for a finitely generated group $G$ that having all
infinite cyclic subgroups undistorted implies all finitely generated
abelian subgroups are undistorted then that short argument could also
be used to establish property (ii) in Theorem \ref{know2}. However
this does not hold in general, for instance some groups of the form
$G=\Z^2\rtimes\Z$ provide counterexamples.  
Consequently in this section we examine what further
conditions can be placed on a finitely generated group in order to
ensure that all of its finitely generated
abelian subgroups are undistorted. This is achieved by examining the
work of G.\,Connor (in \cite{consol} and other papers cited there)
on translation length.

We proceed as follows: if $G$ is any finitely generated group then
put the word length $l_S$ on $G$ (with respect to
some finite generating set $S$). Next let $\tau$ be the associated translation
length function from $G$ to $[0,\infty)$,
that is $\tau(g)=\mbox{lim}_{n\rightarrow\infty}l_S(g^n)/n$.
Now having $\tau(g)>0$ for all infinite order elements $g$ is equivalent
to saying that every cyclic subgroup of $G$ is undistorted. This suggests
the following definition:\\
\begin{defn} We say a finitely generated group $G$ has
{\bf uniformly undistorted cyclic subgroups} if
there exists $c>0$ with $\tau(g)\geq c$ for
all infinite order
$g\in G$.
\end{defn}
Note that, just as for the property of having undistorted cyclic
subgroups, having uniformly undistorted cyclic subgroups is
invariant under change of generating set $S$ (though the
constant $c$ varies), because both $l_S$ and
$\tau$ will be replaced by Lipschitz equivalent functions.

Our result is as follows.
\begin{thm} \label{trids}
Let $G$ be a finitely generated group with uniformly undistorted
cyclic subgroups.
Then any finitely generated abelian subgroup $A$ of $G$ is
undistorted in $G$.
\end{thm}
\begin{proof}
Any finitely generated abelian subgroup will have a finite index subgroup $A$
which is isomorphic to some free abelian group
$\Z^m$ and it is enough to show
that $A$ is undistorted in $G$. First pick out 
any free abelian basis $a_1,\ldots ,a_m$ for $A$.
Now by an $\R$-norm $||\cdot ||$ on $\R^m$ we mean the standard definition from
normed vector spaces, that is it satisfies the triangle inequality with
$||v||$ being zero if and only if $v$ is the zero vector, and also
$||\lambda v||=|\lambda|\cdot ||v||$ for $|\cdot |$ the usual modulus
on $\R$. We will also define a $\Z$-norm on $\Z^m$ to be a function
$f:\Z^m\rightarrow [0,\infty)$ having the same properties, except
the last becomes $f(na)=|n|\cdot||a||$ for all $n\in\Z$ and $a\in\Z^m$.
We also have $\R$- and $\Z$-seminorms where we remove the $||\cdot ||=0$
implies $\cdot=0$ condition.

On considering word length $l_S$ on $G$ with respect to some finite
generating set $S$ and the associated translation length $\tau$, we have
$0\leq\tau(g)\leq l_S(g)$ for any $g\in G$ by repeated use of
the triangle inequality and also 
$\tau(g^n)=|n|\tau(g)$. We further have $\tau(gh)\leq\tau(g)+\tau(h)$
for commuting elements $g,h$ (but not in general). Thus on
restricting $\tau$ to the abelian subgroup $A$ we see that $\tau$ is
 a $\Z$-seminorm on $A$, hence also a $\Z$-norm as
$\tau(a)>0$ for all $a\in A\setminus\{id\}$. Indeed $\tau$ is actually
a discrete $\Z$-norm, given that we have $c>0$ with
$\tau(a)\geq c$ for all $a\in A\setminus\{id\}$.

On regarding $A\cong\Z^m$ as embedded in $\R^m$ via the integer lattice
points, we can extend $\tau$ to $\Q^m$ by dividing through and to
$\R^m$ by taking limits, so that $\tau$ is also an $\R$-seminorm
on $\R^m$. However it is not too hard to see that $\tau$ is in fact
a genuine $\R$-norm, as explained carefully in \cite{step}, and we
denote this by $||\cdot ||_\tau$.

Now take $a=a_1^{n_1}\ldots a_m^{n_m}\in A$ which is also the lattice
point $(n_1,\ldots ,n_m)\in\R^m$. We have the $\ell_1$ norm 
$||\cdot ||_1$ on $\R^m$
but all norms on $\R^m$ are equivalent, so there is $k>0$ such that
\[l_S(a)\geq\tau(a)=||(n_1,\ldots ,n_m)||_\tau
\geq k||(n_1,\ldots ,n_m)||_1=k(|n_1|+\ldots +|n_m|)\]
so $A$ is undistorted in $G$.
\end{proof}

\subsection{$Out(F_n$) and related results}
We finish with an application of the above theorem to the
outer automorphism group $Out(F_n)$ of the rank $n$ free group, as well as
to some related groups. Now, at least for $n\geq 4$, 
$Out(F_n)$ is not linear over any field by \cite{fp} nor
is it a weak CAT(0) group by \cite{ger}. Thus 
we cannot apply earlier results directly to $Out(F_n)$ in order to
establish that every finitely generated abelian subgroup of $Out(F_n)$
is undistorted, which was recently proven in \cite{wrig}. 
But previously the paper \cite{alib}, which
ostensibly had shown that all cyclic subgroups of $Out(F_n)$ are undistorted,
actually gave us more:
\begin{thm} (\cite{alib} Theorem 1.1)\\
Every infinite order element $g$ of $Out(F_n)$ has positive translation
length $\tau(g)$.
Furthermore there exists a constant $c_n>0$ such that $\tau(g)\geq c_n$
for all $g\in Out(F_n)$ (using the generating set consisting of
permutations, inversions and Nielsen twists), so that $Out(F_n)$ has
uniformly undistorted cyclic subgroups.
\end{thm}
Thus combining that (short) paper with Theorem \ref{trids},
we immediately obtain a quick proof of \cite{wrig}:
\begin{co}
Every finitely generated abelian subgroup of $Out(F_n)$ is undistorted.
\end{co}

We end by pointing out that $Out(F_n)$ having
undistorted abelian subgroups can be used to
establish some further consequences.       
\begin{co} If $A$ is any abelian subgroup of the automorphism
group $Aut(F_n)$, or of any free by cyclic group $G=F_n\rtimes_\alpha\Z$
for $\alpha\in Aut(F_n)$ then $A$ is finitely generated and
undistorted in $Aut(F_n)$ or in $G$.
\end{co}
\begin{proof}
If $S,H,G$ are all finitely generated groups with $S\leq H\leq G$ and
$S$ is undistorted in $G$
then $S$ is undistorted in $H$ (else extend the generating set of $H$ to
one of $G$, whereupon the distortion persists). The converse also holds
if $H$ has finite index in $G$.
Now an observation dating
back to Magnus is that $Aut(F_n)$ embeds in $Out(F_{n+1})$ by
considering automorphisms of $F_{n+1}$ which fix the last element of the
basis.

Moreover for the free by cyclic group
$G=F_n\rtimes_\alpha\Z$, we have that if $\alpha$
has infinite order in $Out(F_n)$ then $G$ embeds in $Aut(F_n)$. This
can be seen on taking the copy $F_n$ of inner automorphisms in $Aut(F_n)$
and then $G$ is isomorphic to $\langle \alpha,F_n\rangle\leq Aut(F_n)$.
If however $\alpha$ has finite order then $G$ contains the finite index
subgroup $H=F_n\times\Z$ and this certainly has undistorted finitely
generated abelian subgroups.
\end{proof}

\end{document}